\documentclass[11pt]{scrartcl} 

\usepackage{mathtools}
\mathtoolsset{showonlyrefs}

\usepackage[utf8]{inputenc} 
\usepackage[a4paper, total={6in, 10in}]{geometry}
\usepackage{geometry} 
\usepackage{xcolor}
\usepackage{graphicx} 

\usepackage{booktabs} 
\usepackage{array} 
\usepackage{paralist} 
\usepackage{verbatim} 
\usepackage{subfig} 

\newcommand{\sign}{\mathrm{sign}}

\usepackage{amsmath}
\usepackage{amsthm}
\usepackage{amsfonts}
\usepackage{amssymb}

\usepackage[round]{natbib}

\newtheorem{theorem}{Theorem}[section]

\newtheorem{example}[theorem]{Example}

\newtheorem{proposition}[theorem]{Proposition}
\newtheorem{remark}[theorem]{Remark}

\usepackage{fancyhdr} 
\usepackage{sectsty}
\allsectionsfont{\sffamily\mdseries\upshape} 

\usepackage[nottoc,notlof,notlot]{tocbibind} 
\usepackage[titles,subfigure]{tocloft} 

\title{Three essays on stopping}

\author{Eberhard Mayerhofer\thanks{University of Limerick, Department of Mathematics and Statistics, V94TP9X, Limerick, Ireland, email {eberhard.mayerhofer@ul.ie}. For discussions on this paper I think John Appleby, Florin Avram, Friedrich Hubalek, and Andreas Kyprianou.}}

\begin{document}
\maketitle

\vspace{-1cm}
%

\smallskip

\begin{abstract}

First, we give a closed-form formula for first passage time of a reflected Brownian motion with drift. This corrects a formula by \cite{perry2004first}.

Second, we show that the maximum before a fixed drawdown is exponentially distributed for any drawdown threshold, if and only if the diffusion characteristic $\mu/\sigma^2$ is constant. This complements the sufficient condition formulated by \cite{lehoczky1977formulas}.

Third, we give an alternative proof for the fact that the maximum before a fixed drawdown threshold is exponentially distributed for any spectrally negative L\'evy process, a result due to \cite{mijatovic2012drawdown}.
\end{abstract}

\tableofcontents
\newpage

This paper comprises three essays on stopping.

In section 1, we compute the Laplace transform of the first hitting time of a fixed upper barrier for a reflected Brownian motion with drift. This expands on, and corrects a result by \cite{perry2004first}.

In section 2 we show, by using an intrinsic delay differential equation, that for a diffusion process, the maximum before a fixed drawdown threshold is generically exponentially distributed, only if the diffusion characteristic $\mu/\sigma^2$ is constant. This complements the sufficient condition formulated by \cite{lehoczky1977formulas}. We further construct diffusions, where the exponential law only holds for specific drawdown sizes.

Section 3 uses \cite{lehoczky1977formulas}'s argument to show that the maximum before a fixed drawdown threshold is exponentially distributed for any spectrally negative L\'evy process, the parameter being the right-sided logarithmic derivative of the scale function. This yields an alternative proof to the original one in \cite{mijatovic2012drawdown}.

\section{The first hitting time for a reflected Brownian motion with drift}
Let $X$ be a reflected Brownian motion on $[0, \infty)$, with drift $\mu$ and volatility $\sigma$. By \cite{graversen2000extension} the RBM($\mu,\sigma^2$) can be realized as $\vert \xi_t^x\vert$,
where $\xi^x$ is the unique strong solution of 
\[
d\xi_t=\mu\,\sign(\xi_t)dt+\sigma \,dB_t, \quad \xi_0=x,
\]
where $B$ is a standard Brownian motion.\footnote{This is the generalization of L\'evy's result for (driftless) reflected Brownian motion, which states that RBM($\mu=0,\sigma^2=1$) is equal in law to $\vert x+B\vert $.)} We therefore assume, in the following, a filtered probability space given that supports $B$, and identify $X=(X_t)_{t\geq 0}$ with the $X_t=\vert \xi_t^x\vert$, $x\geq 0$. By Tanaka's formula, we have
\[
X_t=\vert \xi_t\vert=x+\int_0^t \sign(\xi_t)d\xi_t+L_t^0(X)=x+\mu t+\sigma\int_0^t \sign(\xi_t)dB_t+L_t^0(X),
\]
where $L^0(X)$ is the local time of $X$ at $0$. Since the latter is supported on $\{X=0\}$, It\^o's formula implies for any $f\in C^2_b((0,\infty))$, for which $f'(0+)=0$, the process
\[
f(X_t)-f(x)-\int_0^t\mathcal Af(X_s)ds,
\]
is a martingale, where $\mathcal A$ is the differential operator, defined by $\mathcal Af=\frac{\sigma^2}{2}f''(x)+\mu f'(x)$.\footnote{In the language of linear diffusions \cite{borodin2012handbook}, $X$ has infinitesimal generator $\mathcal A$ acting on $\mathcal D(\mathcal A)=\{f\in C^2_b((0,\infty)\mid f'(0+)=0\}$.}

Since, before reaching the boundary $0$, the process cannot be distinguished from a Brownian motion with drift, for $0<\delta+x<x$, the first hitting time
\begin{equation}\label{eq: stopping time}
\tau_\delta:=\inf\{t\geq 0\mid X_t=\delta+x\}
\end{equation}
equals, in distribution, to the first hitting time $\tau_\delta$ of a Brownian motion with drift, starting at $x$. 

Therefore, we may confine ourselves to computing $\tau_\delta$ for barriers $\delta+x$, where $\delta>0$. Our aim is to compute the Laplace transform
\[
\Psi(\theta;\delta,x):=\mathbb E[e^{-\theta \tau_\delta}\mid X_0=x],\quad \theta\geq 0.
\]

\begin{theorem}\label{th: main}
For $\delta\geq 0$, the Laplace transform of the first hitting time of a reflected Brownian motion with drift $\mu$ and volatility $\sigma$
is given by
\begin{equation}\label{eq: LT hitting timex}
\Psi(\theta; x,\delta):=e^{\frac{\delta\mu}{\sigma^2}}\frac{\sqrt{\mu^2+2\theta\sigma^2}\cosh\left(\frac{x\sqrt{\mu^2+2\theta\sigma^2}}{\sigma^2}\right)+\mu \sinh\left(\frac{x\sqrt{\mu^2+2\theta\sigma^2}}{\sigma^2}\right)}{\sqrt{\mu^2+2\theta\sigma^2}\cosh\left(\frac{(x+\delta)\sqrt{\mu^2+2\theta\sigma^2}}{\sigma^2}\right)+\mu \sinh\left(\frac{(x+\delta)\sqrt{\mu^2+2\theta\sigma^2}}{\sigma^2}\right)}
\end{equation}
\end{theorem}
\begin{proof}
Pick $\Phi\in C_c^\infty(\mathbb R)$ such that $\Phi(\xi)=1$ for $\vert \xi\vert\leq x+\delta$. Furthermore, let $\kappa>0$, then for any $\theta\geq 0$ and $t\geq 0$, the function
\[
F(t,x):=e^{-\theta t} \Phi(x)\left(e^{-\kappa x}+\kappa x\right)
\]
satisfies $f:=F(t,\cdot)\in C^2_b$ and $f'(0)=0$.  According to the introductory notes of this section, the process $F(t,X_t)-\int_0^t \partial_s F(s,X_s)ds-\int_0^t \mathcal A F(s,X_s)ds$ is a uniformly bounded martingale, and therefore also the stopped process
\begin{align*}
&F(t, X_{t\wedge \tau_\delta})-(e^{-\kappa x}+\kappa x)-\int_0^{t\wedge \tau_\delta} \partial_t F(s,X_s)ds-\int_0^{t\wedge \tau_\delta} \mathcal A F(s,X_s)ds
\end{align*}
is a true martingale, which starts at zero, $\mathbb P^x$-almost surely. Using the fact that $\Phi(X_{t\wedge \tau_\delta})=1$, we find that the stopped process satisfies for any $t\geq 0$,
\begin{align*}
&e^{-\theta (t\wedge \tau_\delta)} \left(e^{-\kappa X_{t\wedge \tau_\delta}}+\kappa X_{t\wedge \tau_\delta}\right)-(e^{-\kappa x}+\kappa x)+\theta\int_0^{t\wedge \tau_\delta} e^{-\kappa X_s-\theta s} ds+\theta \kappa\int_0^{t\wedge \tau_\delta} e^{-\theta s}X_s ds\\&\qquad-\mu\int_0^{t\wedge \tau_\delta} \left( \kappa e^{-\theta s}-\kappa e^{-\kappa X_s-\theta s} \right)ds-\frac{\sigma^2\kappa^2}{2}\int _0^{t\wedge \tau_\delta} e^{-\kappa X_s-\theta s} ds\\&\qquad=
e^{-\theta (t\wedge \tau_\delta)} \left(e^{-\kappa X_{t\wedge \tau_\delta}}+\kappa X_{t\wedge \tau_\delta}\right)-(e^{-\kappa x}+\kappa x)+\theta \kappa\int_0^{t\wedge \tau_\delta} e^{-\theta s}X_s ds\\&\qquad-\frac{\mu\kappa}{\theta}\left(1-e^{-\theta (t\wedge \tau_\delta)}\right)-\left(\frac{\sigma^2 \kappa^2}{2}-\kappa \mu-\theta\right)\int_0^{t\wedge \tau_\delta} e^{-\kappa X_s-\theta s} ds.
\end{align*}
Letting $t\rightarrow \infty$, we thus get by optional sampling,
\begin{align*}
&(e^{-\kappa (x+\delta)}+\kappa (x+\delta))\mathbb E^x[e^{-\theta \tau_\delta}]-(e^{-\kappa x}+\kappa x)+\theta \kappa\mathbb E^x\left[\int _0^{\tau_\delta} e^{-\theta s}X_s ds\right] \\&-\frac{\mu \kappa}{\theta}(1-\mathbb E^x[e^{-\theta \tau_\delta}])-\left(\frac{\sigma^2 \kappa^2}{2}-\kappa \mu-\theta\right)\mathbb E^x\left[\int_0^{-\tau_\delta} e^{-\kappa X_s-\theta s} ds \right]=0.
\end{align*}
For the two choices $\kappa\in \{\kappa_-,\kappa_+\}$, where
\begin{equation}\label{eq: superalpha}
\kappa_\pm:=\frac{\mu\pm\sqrt{\mu^2+2\theta\sigma^2}}{\sigma^2},
\end{equation}
we thus obtain two equations, for two unknown moments,
\[
\left(e^{-\kappa_\pm(x+\delta)}+\kappa_\pm (x+\delta)+\frac{\mu \kappa_\pm}{\theta}\right)\mathbb E^x[e^{-\theta\tau_\delta}]+\theta\kappa_\pm \mathbb E^x\left[\int _0^\tau e^{-\theta s}X_s ds\right]= (e^{-\kappa_\pm x}+\kappa x)+\frac{\mu\kappa_\pm}{\theta}.
\]
Solving this linear system for the involved moments yields the Laplace transform
of $\tau_\delta$, equation \eqref{eq: LT hitting timex}.
\end{proof}
\subsubsection{Sanity Check: driftless case}
For a first ``sanity check" of Theorem \ref{th: main}, we compute the LT \eqref{eq: LT hitting timex}
independently when $\mu=0$ and $x=0$. In this case, the reflected Brownian motion is equal to $\vert \sigma B\vert$ in law, where $B$
is a standard Brownian motion. But then $\tau_\delta$ equals, in distribution, to
\[
\widetilde\tau_\delta:=\inf\{s>0\mid B_s\in \{\pm\frac{\delta}{\sigma}\}\}.
\]
Now it is well known that the Laplace transform of $\widetilde\tau_\delta$ is given by
\begin{equation}\label{eq: driftless RBM}
\mathbb E^x[e^{-\widetilde\theta \tau_\delta}]=\frac{1}{\cosh(\frac{\delta}{\sigma}\sqrt{2\theta})}
\end{equation}
which indeed coincides with \eqref{eq: LT hitting timex} for $\mu\rightarrow 0$ (that is, zero Sharpe Ratio, zero absorption and null killing).

\subsubsection{Remarks on \cite{perry2004first}}\label{subsec: perry}
\cite[Formula (5.2)]{perry2004first} state a different Laplace transforms than our Theorem \ref{th: main}. Letting $\mu\rightarrow 0$
in \cite[Formula (5.2)]{perry2004first} indeed yields $(\sigma^2=1)$
\[
\mathbb E^x[e^{-\theta\tau_\delta}]=\frac{1}{\cosh(\delta\sqrt{\theta})}
\] 
which contradicts \eqref{eq: driftless RBM}. The proof of \cite[Lemma 5.1]{perry2004first}
can however not be rectified, by merely fixing the (obviously) missing factor of $1/2$. Indeed, in the second line of their proof, they forget a factor $e^{-\kappa W(s)}$ in the second integrand, and thus by inserting special values of $\kappa$ into the process in line 2, one does not get rid of the local-time term, as claimed. 
\section{Diffusions with exponentially distributed gains before fixed drawdowns}
Let $X$ be a diffusion process on the $[-a,\infty)$, satisfying the SDE
\begin{equation}\label{eq: tara}
dX_t=\mu(X_t)dt+\sigma(X_t)dW_t,\quad X_0=0,
\end{equation}
where $\mu(x)$ and $\sigma(x)$ are locally Lipschitz continuous functions of linear growth on $[-a,\infty)$, and $\sigma(x)>0$ thereon.

For a threshold $0<\delta\leq a$, we define $M^\delta$ as the maximum of $X$, prior to a drawdown of size $\delta$, that is
\[
M^\delta=M(\tau^\delta),\quad\text{where}\quad M(t):=\max_{s\leq t}X_s,\quad\text{and}\quad \tau^\delta:=\inf\{t>0\mid M_t-X_t=\delta\}.
\]
We use the abbreviation $\Phi(x):=e^{-2\int_0^x \gamma(u)du}$, where $\gamma(x)=a(x)/\sigma^2(x)$. The following is due to \cite{lehoczky1977formulas}:
\begin{proposition}\label{Lehozky}
\begin{equation}\label{eq: overshoot law}
\log\mathbb P[M^\delta\geq \xi]=-\int_0^\xi \frac{\Phi(u)}{\int _{u-\delta}^u \Phi(s)ds} du,\quad \xi\geq 0.
\end{equation}
\end{proposition}
Caution is needed when interpreting the original paper \cite{lehoczky1977formulas}: Lehoczky uses the letter ``a" for three different objects: The drift $\mu(x)$ is denoted as $a(x)$, while $a$ is the left endpoint of the interval of the support of $X$; third, the threshold $\delta$ in his paper is also called $a$. An inspection of Lehozky's proof reveals that our more general version with $\delta\leq a$ holds.

In terms of diffusion characteristics, Lehoczky's result holds in a more general context. First, the assumption of locally Lipschitz coefficients are too strong, and can be relaxed. For example, we can relax to H\"older regularity of $\sigma(x)$ of order no worse than $1/2$, due to \cite{yamada1971uniqueness}. Also, we can allow reflecting or absorbing boundary conditions, thus include reflected diffusions. For instance, Proposition \ref{Lehozky} holds for a Brownian motion with drift, starting at $0$ and being reflected at $-a$, because, the process $X$ cannot hit $-a$, before it reaches a strictly positive maximum, due to strict positive volatility $\sigma(0)>0$.

From \eqref{eq: overshoot law} it can be seen that when $\mu/\sigma^2$ is constant, $M^\delta$ is exponentially distributed (the special case for for a Brownian motion with drift is due to \cite{Taylor}, and independently discovered by \cite{golub2016multi}). \cite{mijatovic2012drawdown} extended this result to spectrally negative L\'evy processes: For those, $M^\delta$ is also exponentially distributed, with the parameter being the right-sided logarithmic derivative of the scale function, evaluated at the drawdown threshold.

This section characterizes the exponential law for diffusions:
\begin{theorem}\label{th: mainx}
The following are equivalent:
\begin{enumerate}
\item \label{x1} $\mu(x)/\sigma^2(x)$ is a constant on $[-a,\infty)$.
\item For each $\delta>0$, $M^\delta$ is exponentially distributed.
\end{enumerate}
\end{theorem}
\begin{proof}[Proof of the Theorem]
Sufficiency of the first condition for the second one follows directly from Proposition \ref{Lehozky}. 
Suppose, therefore that for each $0<\delta\leq a$, there exists $\Lambda(\delta)>0$ such that $M^\delta\sim \mathcal E(\Lambda(\delta))$. Then, due to \eqref{eq: overshoot law},
\begin{equation}\label{eq: basic}
\int_0^\xi \frac{\Phi(u)}{\int _{u-\delta}^u \Phi(s)ds} du=\Lambda(\delta)\xi,\quad \xi\geq 0,\quad \delta\leq a.
\end{equation}
By this particular functional form, and, since $\mu/\sigma^2$ is continuous, it follows that the functions $\Lambda(\delta)$
and $\Phi(x)$ are continuously differentiable. By differentiating \eqref{eq: basic} with respect to $\xi$, we have
\[
\Phi(\xi)=\Lambda(\delta)\int_{\xi-\delta}^\xi\Phi(u)du,\quad \xi\geq 0,\quad \delta\leq a,
\]
and differentiating with respect to $\delta$ yields, in conjunction with the previous identity,
\[
\frac{\Phi(\xi-\delta)}{\Phi(\xi)}=-\frac{\Lambda'(\delta)}{\Lambda^2(\delta)},\quad \xi\geq 0,\quad\delta\leq a.
\]
Therefore, also
\[
\frac{\Phi(\xi)}{\Phi(\xi+\delta)}=-\frac{\Lambda'(\delta)}{\Lambda^2(\delta)},\quad \xi\geq 0,\quad\delta\leq a,
\]
and dividing the last two equations yields Lobacevsky's functional equation\footnote{See \cite[p. 82, Chapter 2 (eq. (16)]{aczel1966lectures} and the references
therein.}
\begin{align}\label{eq: lob fun}
\Phi(\xi-\delta)\Phi(\xi+\delta)&=\Phi(\xi)^2,\quad\xi\geq 0,\quad\delta\leq a,\\
\Phi(0)&=1.
\end{align}
Note, $\Phi$ is continuously differentiable, and strictly positive. Hence, by taking derivatives with respect to $\delta$,
we get
\[
\frac{\Phi'(\xi-\delta)}{\Phi(\xi-\delta)}=\frac{\Phi'(\xi+\delta)}{\Phi(\xi+\delta)},
\]
and by setting $\xi=\delta$, we thus have
\[
\Phi'(2\xi)=\alpha \Phi(2\xi),\quad \Phi(0)=1,\quad 0<\xi\leq a,
\]
where $\alpha=\Phi'(0)/\Phi(0)\in\mathbb R$. We conclude that for some $\beta\geq 0$,
\begin{equation}\label{eq: val sol}
\Phi(\xi)=e^{\beta\xi},\quad 0\leq \xi\leq 2a.
\end{equation}
By \eqref{eq: lob fun} we can extend the exponential solution to $-a\leq \xi<0$: By setting $\xi=0$, we indeed have
\[
\Phi(-\delta)=\frac{\Phi^2(0)}{\Phi(\delta)}=\frac{1}{e^{\beta\delta}}=e^{-\beta\delta}, \quad 0<\delta\leq a.
\]
Similarly, we can succesively extend the validity of \eqref{eq: val sol} to the right, using the functional equation
\eqref{eq: lob fun}. Now that $\Phi(\xi)=e^{\beta \xi}$ for some $\beta\geq 0$, we have, by taking the logarithmic derivative of $\Phi$,
that $\mu(x)/\sigma^2(x)$ is indeed a constant.
\end{proof}
Examples of processes for which the running maximum at drawdown is exponentially distributed, are the following:
\begin{enumerate}
\item \label{a} ($a=-\infty$): Brownian motion with drift $\sigma B_t+\mu t$.
\item \label{b} $(a<\infty$): Reflected Brownian motion with drift, reflected at $-a$,
\item Similar examples as in \ref{a} and \ref{b} can be constructed, where $\mu(x)/\sigma^2(x)$ is constant. These include reflected diffusions.
\end{enumerate}
However, there are processes that do not satisfy Theorem \ref{th: mainx}, but exhibit exponentially distributed gains before $\delta$ drawdowns for specific choices of $\delta$. One can, for instance, let $\mu/\sigma^2$ be constant only on $[-1,\infty)$, and modify $\mu,\sigma^2$ on $[-2,-1)$
in such a way, that the SDE \eqref{eq: tara} has unique global strong solution (or, alternatively, make $-a$ an absorbing boundary). Then, by Proposition \ref{Lehozky}, for any $\delta<1$ the 
maximum at drawdown of size $\delta$ is exponentially distributed. It goes without saying, that there must exist $\delta>1$ for which this is not the case.

More sophisticated examples can be constructed by solving delay differential equations for $\Phi=e^{-2\int_0^\xi \mu(u)/\sigma^2(u)du}$ for individual thresholds $\delta$. E.g., for $\Lambda(\delta)=\delta=1$, and $\Phi(x)=(x+2)/2$ on $[-1,0]$, one can solve the equation
\[
\Phi'(\xi)=\Phi(\xi)-\Phi(\xi-1),\quad \xi\geq 0
\]
subject to $\Phi(\xi)=(\xi+2)/2$ for $\xi\in[-1,0]$. This problem has a unique solution with exponential growth. Due to the initial data, it
cannot be exponential, though.
\section{Lehoczky's proof for spectrally negative L\'evy martingales}
We study in this section the distribution of maximal gains\footnote{This random gain is called ``overshoot" in \cite{golub2016multi}. In this section, we refrain from using this terminology due to its established meaning in the field of L\'evy processes - it is the discrepancy between a certain threshold, and a jump processes' value, passing beyond that threshold.} of processes, prior to the occurrence of a fixed loss $\delta>0$. \cite{golub2016multi,golub2} claim that for a Brownian motion (the toy model of a fair game), this gain is exponentially  distributed, with parameter $\delta$; thus in average, one gains $\delta$ before experiencing a loss of size $\delta$. This result is independent of the volatility of the Brownian motion. In private communication, \cite{golubprivate} raised the question, of whether similar scaling laws hold for other processes, e.g., other diffusion models, or processes with jumps. Such models are useful as benchmark models in the context of certain event-based high-frequency trading algorithms, where the Brownian motion is used as a proxy for an asset, and the location of the maximum suggests the beginning of a trend reversal.\footnote{It goes without saying that the first time, this maximum is attained, is not a stopping time; otherwise one could devise arbitrage strategies that short-sell the asset at the maximum.}

The conjecture that a fair game in average experiences the exact same gain, as is lost later on, may appear intuitive. And this is indeed the case for many continuous-time martingales, those who are time-changed Brownian motions, with a quadratic variation tending to infinity, along almost every path (because the timing is not relevant here). But it is not true for L\'evy martingales, as can seen from Theorem \ref{th: main}. Nevertheless the (exponential) distribution of gains, not its parameter, is universal within the class of spectrally negative L\'evy processes. Besides, the martingale property is not needed to arrive at this result.

After Theorem \ref{th: main} was proved in summer 2019, F. Hubalek kindly pointed out that the result is, in identical form, preceded by \cite{mijatovic2012drawdown}. Our proof is, however, similar to the one of
\cite{lehoczky1977formulas}, and is therefore an alternative, and simpler one. (Finally, we also found a replication of Lehoczky's proof in \cite[Lemma 3.1]{leho3}, however, also this proof is more difficult than ours, due the general discretization used therein).

We assume, that a L\'evy process $X$ is given with downward jumps only, but not equals the negative of a L\'evy subordinator\footnote{This is the natural non-degeneracy condition of \cite[Chapter VII]{Bertoin}, to ensure that the process creeps up to any level.}. Such a process is defined by its L\'evy exponent 
\[
\Psi(\theta):=\frac{1}{t}\log\mathbb E[e^{\theta X_t}],\quad \theta>0,
\]
which is assumed to be of the form
\[
\Psi(\theta)=\mu \theta+\frac{\sigma^2\theta^2}{2}+\int_{(-\infty, 0)}\left(e^{\theta \xi}-1-\theta \xi 1_{[-1,0)}(\xi)\right)\nu(d\xi), \quad\theta>0,
\]
with L\'evy-Khintchine triplet $\mu\in\mathbb R$, $\sigma\in\mathbb R$ and a sigma-finite measure $\nu(d\xi)$ supported
on $(-\infty,0)$, integrating $\xi^2$ near $0$.

The scale function $W$, is the unique absolutely continuous function $[0,\infty)\rightarrow [0,\infty)$
with Laplace transform
\begin{equation}
\int_0^\infty e^{-\theta x}W(x)dx=\frac{1}{\Psi(\theta)},\quad \theta>0.
\end{equation}
Since the processes lack positive jumps, they can only creep up. This assumption is essential to
obtain exit probabilities from compact intervals, and also for the main Theorem \ref{th: main}.
\begin{theorem}\cite[Theorem VII.8]{Bertoin}\label{th: super}
Let $x,y>0$, the probability that $X$ makes its first exit from $[-x,y]$ at $y$ is 
\begin{equation}
\mathbb P[\tau_y<\tau_{-x}]=\frac{W(x)}{W(x+y)}.
\end{equation}
\end{theorem}
We are ready to state and proof the main theorem:
\begin{theorem}\label{th: main}
For a spectrally negative process, not equals to a negative subordinator, the maximal gain $M^\delta$ before a $\delta$-loss is exponentially distributed with parameter equals the logarithmic derivative of the scale function, that is,
\[
\mathbb P[M^\delta\geq \xi]=e^{-\frac{W'(\delta+)}{W(\delta)}\xi}.
\]
\end{theorem}
\begin{proof}[Proof of Theorem \ref{th: main}]
The proof is inspired by \cite{golub2016multi}, however, the exact same idea can be traced back to \cite{lehoczky1977formulas} in the general context of univariate diffusions processes. Let $A_{k,n}$ be the event that $X$ reaches $k\xi/n$ before $-\delta+(k-1)/2^n \xi$ ($k=1,\dots,2^n$). Then
Then $M^\delta\geq \xi$ can be approximated by $\bigcap_{k=1}^n A_{k,n}$, which are decreasing for increasing $n$. In other words,
\[
\{M^\delta\geq \xi\}=\bigcap_{n=1}^\infty\bigcap_{k=1}^{2^n} A_{k,n}
\]
Therefore
\[
\mathbb P[M^\delta\geq \xi]=\lim_{n\rightarrow\infty} \mathbb P\left[\bigcap_{k=1}^{2^n} A_{k,n}\right].
\]
Due to state-independence of the process (translation invariance) and the Markov property
\[
\mathbb P[A_{1,n}]\times \prod_{k=2}^{2^n}\mathbb P[A_{k,n}\mid A_{k-1,n} ]=(\mathbb P[A_{1,n}])^{2^n}=\left(\frac{W(\delta)}{W(\delta+\xi/2^n)}\right)^{2^n},
\]
where the last identity follows from Theorem \ref{th: super}. Since $W$ is differentiable from the right at $\delta$, applying L'Hospital's rule yields
\[
\log\mathbb P[M^\delta\geq \xi]=\lim_{n\rightarrow\infty}\log (\mathbb P[A_{1,n}])^{2^n}=-\xi \frac{W'(\delta+)}{W(\delta)}.
\]
\end{proof}
\begin{remark}
Theorem \ref{th: main} implicitly requires right-differentiability of the scale functions, which is for free, because it can be rewritten as an integral
of the tail of some finite measure, see \cite[Chapter VII]{Bertoin}. However, in many models, full $C^1$-regularity is guaranteed (cf.~\cite[Lemma 2.4]{kuznetsov2012theory}).
\end{remark}
\subsection{Examples}
The scale functions for the below processes are taken from review article of \cite{hubalek}.
\begin{example}[Compound Poisson Process]
Assume we have a compound Poisson process with negative exponentially distributed jumps,
\[
X_t=ct-\sum_{k=0}^{N_t^\lambda}\xi_k,\quad \xi_k\quad \text{i.i.d. and }\sim \mathcal E(\mu),\quad c-\lambda/\mu>0.
\]
We get
\[
W(x)=\frac{1}{c}\left(1+\frac{\lambda}{c\mu-\lambda}(1-e^{-(\mu-\lambda/c)x})\right).
\]
Clearly $W\in C^1(0,\infty)$, 
\[
W'(x)=\frac{\lambda}{c^2}e^{-(\mu-\lambda/c)x}.
\]
Therefore, by Theorem \ref{th: main}
\[
M^\delta\sim \mathcal E\left(\frac{\lambda/c}{e^{\delta(\mu-\lambda/c)}-\frac{\lambda/c}{\mu-\lambda/c}}\right),\quad\lim_{\delta\downarrow 0}\mathbb E[M^\delta]=\lambda/c>0,\quad \lim_{\delta\uparrow \infty}\mathbb E[M^\delta]=\mu-\lambda/c<\infty.
\]
\end{example}
Unlike the previous example, the following two examples exhibit the same qualitative dependence on the threshold $\delta$, as the standard Brownian motion, where $M^\delta\sim \mathcal E(1/\delta)$: when $\delta\rightarrow 0$, the average maximum at drawdown of size $\delta$ tends to $0$, and when $\delta\rightarrow\infty$, this average goes to infinity. 
\begin{example}[Brownian motion with drift]
A Brownian motion with drift $\mu>0$ and volatility $\sigma$,
\[
X_t=\mu t+\sigma B_t
\]
has scale function
\[
W(x)\sim e^{-\mu x/\sigma^2}\sinh(\sqrt{\mu} x/\sigma^2).
\]
hence
\[
\frac{W'(x)}{W(x)}=\frac{-\mu/\sigma^2 \sinh(\sqrt{\mu} x/\sigma^2) +\sqrt\mu/\sigma^2\cosh(\sqrt{\mu} x/\sigma^2)}{\sinh(\sqrt{\mu} x/\sigma^2).}
\]
Therefore, by Theorem \ref{th: main} (see e.g., \cite{golub2016multi}),
\[
M^\delta\sim \mathcal E\left(\mu/\sigma^2\left(\coth(\sqrt{\mu} \delta/\sigma^2)-1\right)\right).
\]
\end{example}
\begin{example}[\cite{caballero2006conditioned}]
This is a L\'evy process without diffusion component, defined by its L\'evy measure
\[
\nu(d\xi)=\frac{e^{(\beta-1)\xi}}{(e^\xi-1)^{\beta+1}},\quad \xi<0,
\]
where $\beta\in (1,2)$, and its Laplace exponent,
\[
\Psi(\theta)=\frac{\Gamma(\theta+\beta)}{\Gamma(\theta)\Gamma(\beta)},\quad\theta>0.
\]
The process exhibits Infinite variation jumps, and drifts to $-\infty$, because $\Psi'(0)<0$. The Scale function is
\[
W(x)=(1-e^{-x})^{\beta-1}
\]
Using Theorem \ref{th: main} we thus get
\[
M^\delta\sim \mathcal E\left(\frac{\beta-1}{e^\delta-1}\right),\quad\mathbb E[M^\delta]=\frac{e^{\delta}-1}{\beta-1}
\]
\end{example}
The asymptotic behaviour of the logarithmic derivative of the scale function of a spectrally negative L\'evy process can be characterized,
using the asymptotic behaviour of $W$ and $W'$, cf.~\cite[Chapter 3]{kuznetsov2012theory}. For instance, $W(0)=W(0+)=0$, if and only if the process is of infinite variation. In the case of finite variation, we can write
the process as $\delta t-J_t$, where $J$ is a subordinator; and then $W(0)=1/\delta>0$. Furthermore, $W'(0+)=\infty$, if a diffusion component is present, or if the L\'evy measure is infinite. These general findings are
consistent with the three examples.
\bibliographystyle{apalike} 

\begin{thebibliography}{}

\bibitem[Acz{\'e}l, 1966]{aczel1966lectures}
Acz{\'e}l, J. (1966).
\newblock {\em Lectures on functional equations and their applications},
  volume~19.
\newblock Academic press.

\bibitem[Bertoin, 1996]{Bertoin}
Bertoin, J. (1996).
\newblock {\em L{\'e}vy processes}, volume 121.
\newblock Cambridge university press Cambridge.

\bibitem[Borodin and Salminen, 2012]{borodin2012handbook}
Borodin, A.~N. and Salminen, P. (2012).
\newblock {\em Handbook of {B}rownian motion-facts and formulae}.
\newblock Birkh{\"a}user.

\bibitem[Caballero and Chaumont, 2006]{caballero2006conditioned}
Caballero, M.~E. and Chaumont, L. (2006).
\newblock Conditioned stable {L}{\'e}vy processes and the {L}amperti
  representation.
\newblock {\em Journal of Applied Probability}, 43(4):967--983.

\bibitem[Golub, 2014]{golubprivate}
Golub, A. (2014).
\newblock Private communication.
\newblock Dublin.

\bibitem[Golub et~al., 2016]{golub2016multi}
Golub, A., Chliamovitch, G., Dupuis, A., and Chopard, B. (2016).
\newblock Multi-scale representation of high frequency market liquidity.
\newblock {\em Algorithmic Finance}, 5(1-2):3--19.

\bibitem[Golub et~al., 2018]{golub2}
Golub, A., Glattfelder, J.~B., and Olsen, R.~B. (2018).
\newblock The alpha engine: designing an automated trading algorithm.
\newblock In {\em High-Performance Computing in Finance}, pages 49--76. Chapman
  and Hall/CRC.

\bibitem[Graversen et~al., 2000]{graversen2000extension}
Graversen, S.~E., Shiryaev, A.~N., et~al. (2000).
\newblock An extension of {P}. {L}{\'e}vy's distributional.
\newblock {\em Bernoulli}, 6(4):615--620.

\bibitem[Hubalek and Kyprianou, 2011]{hubalek}
Hubalek, F. and Kyprianou, E. (2011).
\newblock Old and new examples of scale functions for spectrally negative
  {L}{\'e}vy processes.
\newblock In {\em Seminar on stochastic analysis, random fields and
  applications VI}, pages 119--145. Springer.

\bibitem[Kuznetsov et~al., 2012]{kuznetsov2012theory}
Kuznetsov, A., Kyprianou, A.~E., and Rivero, V. (2012).
\newblock The theory of scale functions for spectrally negative l{\'e}vy
  processes.
\newblock In {\em L{\'e}vy matters II}, pages 97--186. Springer.

\bibitem[Landriault et~al., 2017]{leho3}
Landriault, D., Li, B., Zhang, H., et~al. (2017).
\newblock On magnitude, asymptotics and duration of drawdowns for l{\'e}vy
  models.
\newblock {\em Bernoulli}, 23(1):432--458.

\bibitem[Lehoczky, 1977]{lehoczky1977formulas}
Lehoczky, J.~P. (1977).
\newblock Formulas for stopped diffusion processes with stopping times based on
  the maximum.
\newblock {\em The Annals of Probability}, pages 601--607.

\bibitem[Mijatovi{\'c} and Pistorius, 2012]{mijatovic2012drawdown}
Mijatovi{\'c}, A. and Pistorius, M.~R. (2012).
\newblock On the drawdown of completely asymmetric {L}{\'e}vy processes.
\newblock {\em Stochastic Processes and their Applications},
  122(11):3812--3836.

\bibitem[Perry et~al., 2004]{perry2004first}
Perry, D., Stadje, W., and Zacks, S. (2004).
\newblock The first rendezvous time of {B}rownian motion and compound
  {P}oisson-type processes.
\newblock {\em Journal of applied probability}, 41(4):1059--1070.

\bibitem[Taylor, 1975]{Taylor}
Taylor, H.~M. (1975).
\newblock A stopped {B}rownian motion formula.
\newblock {\em The Annals of Probability}, 3(2):234--246.

\bibitem[Yamada et~al., 1971]{yamada1971uniqueness}
Yamada, T., Watanabe, S., et~al. (1971).
\newblock On the uniqueness of solutions of stochastic differential equations.
\newblock {\em Journal of Mathematics of Kyoto University}, 11(1):155--167.

\end{thebibliography}

\end{document}